\documentclass[12pt,a4paper]{article}

\usepackage[font=small,margin=3cm]{caption}

\usepackage{authblk}
\usepackage[margin=2.2cm]{geometry}
\usepackage{t1enc}
\usepackage[utf8]{inputenc}
\usepackage{amsmath,amsthm,amssymb}
\usepackage{graphicx}
\usepackage{enumerate}
\usepackage{hyperref}
\usepackage{bm}
\usepackage{comment}
\usepackage{amsfonts}
\usepackage{graphicx,caption}
\usepackage{bm}
\usepackage{amsmath, amsthm, amssymb}
\usepackage{graphicx}
\usepackage{hyperref}
\usepackage{relsize}
\usepackage{blkarray}
\usepackage{tabstackengine}
\stackMath

\usepackage[table]{xcolor}

\usepackage{algpseudocode}

\usepackage{bbm}

\theoremstyle{plain}
\usepackage{amsthm}
\makeatletter
\newcommand{\newreptheorem}[2]{\newtheorem*{rep@#1}{\rep@title}\newenvironment{rep#1}[1]{\def\rep@title{#2 \ref*{##1}}\begin{rep@#1}}{\end{rep@#1}}}
\makeatother

\newtheorem{theorem}{Theorem}
\newtheorem*{theorem-non}{Theorem}
\newtheorem*{non-lemma}{Lemma}
\newtheorem{lemma}[theorem]{Lemma}
\newreptheorem{lemma}{Lemma}

\newtheorem{corr}[theorem]{Corollary}

\theoremstyle{definition}
\newtheorem{remark}[theorem]{Remark}

\DeclareMathOperator{\Tr}{Tr}

\begin{document}

\title{The homology torsion growth of determinantal hypertrees}
\author{Andr\'as M\'esz\'aros}
\date{}
\affil{HUN-REN Alfr\'ed R\'enyi Institute of Mathematics, Budapest, Hungary}
\maketitle
\begin{abstract}
Fix a dimension $d\ge 2$, and let $T_n$ be a random $d$-dimensional determinantal hypertree on $n$ vertices. We prove that 
\[\frac{\log|H_{d-1}(T_n,\mathbb{Z})|}{{{n\choose {d}}}}\]
converges in probability to a constant $c_d$, which satisfies
 \[\frac{1}2 \log\left(\frac{d+1}e\right)\le c_d\le \frac{1}2 \log\left(d+1\right) .\]

\end{abstract}

\section{Introduction}

\emph{Determinantal hypertrees} are natural higher dimensional generalizations of a uniform random spanning tree of a complete graph. Fix a dimension $d\ge 2$. A $d$-dimensional simplicial complex $S$ on the vertex set $[n]=\{1,2,\dots,n\}$ is called a $d$-dimensional hypertree, if
\begin{enumerate}[\hspace{30pt}(a)]
 \item\label{pra} $S$ has complete $d-1$-skeleton;
 \item\label{prb} The number of $d$-dimensional faces of $S$ is ${n-1}\choose{d}$;
 \item\label{prc} The homology group $H_{d-1}(S,\mathbb{Z})$ is finite.
\end{enumerate}

In one dimension, a spanning tree must be connected, property \eqref{prc} above is the higher dimensional analogue of this requirement. Among the complexes satisfying both \eqref{pra} and~\eqref{prc}, hypertrees have the minimum number of $d$-dimensional faces. For a graph~$G$, if the reduced homology group $\tilde{H}_0(G,\mathbb{Z})$ is finite, then it is trivial. This statement fails in higher dimensions since for a hypertree $S$, the order of $H_{d-1}(S,\mathbb{Z})$ can range from $1$ to $\exp(\Theta(n^d))$, see \cite{kalai1983enumeration}. Thus, while the homology of spanning trees is uninteresting, the homology of $d$-dimensional hypertrees is a very rich subject to study.

Kalai's generalization of Cayley's formula \cite{kalai1983enumeration} states that
\begin{equation}\label{kalaiformula}\sum |H_{d-1}(S,\mathbb{Z})|^2=n^{{n-2}\choose {d}},\end{equation}
where the summation is over all the $d$-dimensional hypertrees $S$ on the vertex set $[n]$. See also \cite{duval2009simplicial} for various generalizations of this formula. Kalai's formula suggests that the natural probability measure on the set of hypertrees is the one where the probability assigned to a hypertree $S$ is \begin{equation}\label{measuredef}
 \frac{|H_{d-1}(S,\mathbb{Z})|^2}{n^{{n-2}\choose {d}}}.
\end{equation}
It turns out that this measure is a determinantal probability measure \cite{lyons2003determinantal,hough2006determinantal}. Thus, a random hypertree $T_n$ distributed according to \eqref{measuredef} is called a determinantal hypertree. General random determinantal complexes were investigated by Lyons \cite{lyons2009random}. 

We discuss discrete determinantal measures in more detail in Sections~\ref{secgendet}~and~\ref{secdethyp}. Here we only provide a quick description of $T_n$ as a determinantal measure.  Let $P_n$ be the orthogonal projection to the $d$-dimensional coboundaries of the simplex on the vertex set $[n]$. Note that the rows and columns of $P_n$ are indexed by ${{[n]}\choose{d+1}}$. Then, for any ${{n-1}\choose {d}}$ element subset $A$ of~${{[n]}\choose{d+1}}$, we have
\[\mathbb{P}(T_n(d)=A)=\det (P_n[A]),\]
where $T_n(d)$ denotes the set of $d$-dimensional faces of $T_n$, and $P_n[A]$ is the submatrix of $P_n$ determined by the rows and columns with index in $A$. See Lemma~\ref{lemmahypertreeisdet} for details, including a more explicit description of $P_n$.  

While uniform random spanning trees are well-studied \cite{ald1,ald2,ald3,grimmett1980random,szekeres2006distribution,lyons2017probability}, a theory of determinantal hypertrees started to emerge only recently. For $d=2$, Kahle and Newman~\cite{kahle2022topology} proved that with high probability, $T_n$ is apsherical, that is, it has a contractible universal cover. Moreover, with high probability the fundamental group $\pi_1(T_n)$ is hyperbolic and has cohomological dimension $2$. Vander Werf~\cite{werf2022determinantal} provided a useful description of the links of determinantal hypertrees, which makes it possible to apply local-to-global criteria to investigate the expansion properties of determinantal hypertrees. Relying on these criteria, Vander Werf~\cite{werf2022determinantal} proved that for $d=2$ and any $\delta>0$, if we take the union of $\delta\log(n)$ independent copies of~$T_n$, then the fundamental group of the resulting simplicial complex has property~(T) with high probability. Also using local-to-global criteria, the author~\cite{meszaros2023coboundary} proved that given any dimension $d$, there is a constant $k_d$ such that if we take the union of $k_d$ independent copies of~$T_n$, then the resulting simplicial complex is a coboundary expander with high probability. Linial and Peled~\cite{linial2019enumeration} provided bounds on the number of hypertrees. The author determined the local weak limit of determinantal hypertrees~\cite{meszaros2022local}, and for $d=2$, he proved both upper~\cite{meszaros2025bounds} and lower~\cite{meszaros20242} bounds on $H_1(T_n,\mathbb{F}_2)$. See also \cite{meszaros2023cohen,lee2025distribution} for results on random matrix models inspired by determinantal hypertrees. 

In this paper, we investigate the growth rate of $|H_{d-1}(T_n,\mathbb{Z})|$ and prove the following theorem.

\begin{theorem}\label{thm1}
 There is a constant $c_d$ such that
 \[\lim_{n\to\infty} \frac{\log |H_{d-1}(T_n,\mathbb{Z})|}{{{n}\choose{d}}}=c_d\]
 in probability.

 Moreover,
 \[\frac{1}2 \log\left(\frac{d+1}e\right)\le c_d\le \frac{1}2 \log\left(d+1\right) .\]
\end{theorem}

There are strong results on the growth of rational Betti numbers for random complexes and also in the more algebraic setting of towers of finitely sheeted coverings of complexes \cite{linial2016phase,linial2019enumeration,luck1994approximating,farber1998geometry,thom2008sofic,abert2013benjamini}. The understanding of the normalized mod $p$ Betti numbers or the normalized log of the size of the torsion of the integral homology groups is more limited in both the random and the algebraic settings. Mostly, we can only determine the limit of these quantities when they vanish asymptotically (or in the case of mod $p$ Betti numbers when they are asymptotically equal to rational Betti numbers) \cite{linial2006homological,hoffman2017threshold,luczak2018integral,meszaros2025bounds,aronshtam2015does,newman2018integer,meshulam2009homological,kahle2016inside,as1,as2,as3,as4}. It is believed that the homology groups of the Linial--Meshulam complex have trivial torsion parts away from the critical density \cite{luczak2018integral,kahle2020cohen}. It is also conjectured that the torsion in the homology of an arithmetic group 
grows subexponentially with the covolume except for very specific cases~\cite{bergeron2013asymptotic}. 

In contrast to these vanishing results, Theorem~\ref{thm1} provides a positive growth rate.

\subsection{Outline of the proof of Theorem~\ref{thm1}}

Our proof relies on spectral methods. Let $\partial_{T_n,d}$ be the matrix of the $d$-th boundary map of $T_n$. The random probability measure $\mu_n$ is defined as the empirical measure on the eigenvalues of the Laplacian matrix $\partial_{T_n,d}\partial_{T_n,d}^T$, that is,
\[\mu_n=\frac{1}{{{n}\choose{d}}}\sum_{i=1}^{{n}\choose{d}} \delta_{\lambda_i},\]where $\delta_x$ is the Dirac-measure on $x$, and $\lambda_1,\dots,\lambda_{{{n}\choose{d}}}$ are the eigenvalues of $\partial_{T_n,d}\partial_{T_n,d}^T$.

The measure $\mu_n$ is related to $H_{d-1}(T_n,\mathbb{Z})$ by the formula
\begin{equation}\label{intHvsmu}
 \frac{\log|H_{d-1}(T_n,\mathbb{Z})|}{{{n}\choose{d}}}=\frac{1}2 \int \mathbbm{1}(t>0)\log(t) d\mu_n(t)+o(1),
\end{equation}
see Lemma~\ref{Hvslog}. 

The author proved that $T_n$ has a local weak limit \cite{meszaros2022local}. Relying on the techniques of Linial and Peled~\cite{linial2016phase}, this implies that there is a (deterministic) limiting measure~$\mu$ such that
\begin{equation}\label{limitexchange}
 \lim_{n\to\infty}\int f(t) d\mu_n(t)=\int f(t) d\mu(t)
\end{equation}
in probability for all bounded continuous functions $f(t)$. By \eqref{intHvsmu}, it is enough to prove that~\eqref{limitexchange} is also true for $\mathbbm{1}(t>0)\log(t)$ in place of $f(t)$. Proving this would also give the formula 
\[c_d=\frac{1}2\int \mathbbm{1}(t>0)\log(t) d\mu(t)\]
for the limiting constant $c_d$.
 However, $\mathbbm{1}(t>0)\log(t)$ is not bounded. Thus, besides \eqref{limitexchange}, we also need to have good control over the behavior of $\mu_n$ both at $0$ and at $+\infty$. Out of these two, the more challenging problem is to understand the measure $\mu_n$ near~$0$. We control the behavior of $\mu_n$ near $0$ by proving a new general result on the spectrum of certain random matrices associated to discrete determinantal measures, see Section~\ref{secgendet}.

See also ~\cite{lyons2005asymptotic,lyons2010identities,csikvari2016matchings} for slightly different settings where the limiting behavior of the integral of the logarithm function against a sequence of measure was understood. 

\bigskip

\textbf{Acknowledgments.} The author was supported by the NKKP-STARTING 150955 project and the Marie Sk\l{}odowska-Curie Postdoctoral Fellowship "RaCoCoLe". The author is grateful to Mikl\'os Ab\'ert and the anonymous referee for their comments.

\section{Probabilistic results}
\subsection{A few facts from linear algebra}

Given $0\le k\le m$, let us define the elementary symmetric polynomial of degree $k$ in $m$ variables the usual way, that is,
\[\sigma_k(x_1,x_2,\dots,x_m)=\sum_{\substack{A\subset [m]\\|A|=k}}\quad \prod_{i\in A}x_i.\]

Let $M$ be an $m\times m$ matrix. Let $\lambda_1,\lambda_2,\dots,\lambda_m$ be the eigenvalues of $M$.

Assume that the rows and columns of $M$ are both indexed by the set $H$. For $A\subset H$, let $M[A]$ be the square submatrix of $M$ determined by the rows and columns with index in $A$. We use the convention that $\det( M[\emptyset])=1$. For the proof of the following well-known lemma, see, for example, \cite[Theorem 1.2.16]{horn2012matrix}.

\begin{lemma}\label{linearalg}
For any $0\le k\le m$, we have
\[\sigma_k(\lambda_1,\lambda_2,\dots,\lambda_m)=\sum_{\substack{A\subset H\\|A|=k}} \det (M[A]).\]
\end{lemma}

\begin{corr}\label{sigmainv}
 Assume that $\det (M)\neq 0$. For any $0\le k\le m$, we have
 \[\sigma_k\left(\lambda_1^{-1},\lambda_2^{-1},\dots,\lambda_m^{-1}\right)=\sum_{\substack{A\subset H\\|A|=m-k}} \frac{\det (M[A])}{\det(M)}.\]
\end{corr}
\begin{proof}
 The statement follows by combining Lemma~\ref{linearalg} with the identity that
 \[\sigma_k\left(\lambda_1^{-1},\lambda_2^{-1},\dots,\lambda_m^{-1}\right)=\frac{\sigma_{m-k}(\lambda_1,\lambda_2,\dots,\lambda_m)}{\lambda_1\lambda_2\cdots\lambda_m}.\qedhere\]
\end{proof}

\begin{corr}\label{projdetsum}
Assume that $M$ is a projection matrix of rank $r$. Then for any $0\le k\le r$, we have
\[\sum_{\substack{A\subset H\\|A|=k}} \det (M[A])={{r}\choose{k}}.\]
\end{corr}
\begin{proof}
 The eigenvalues of $M$ are $1$ with multiplicity $r$, and $0$ with multiplicity $m-r$. Thus, by Lemma~\ref{linearalg}, we see that
 \[\sum_{\substack{A\subset H\\|A|=k}} \det (M[A])=\sigma_k(\underbrace{0,0,\dots,0}_{m-r\text{ times}},\underbrace{1,1,\dots,1}_{r\text{ times}})={{r}\choose{k}}.\qedhere\]
\end{proof}
The proof of the following well-known lemma can be found in \cite[Theorem 1.3.22.]{horn2012matrix}.
\begin{lemma}\label{switch}
 Let $A$ be any matrix. Then the non-zero eigenvalues of $AA^T$ (with multiplicities) are the same as the non-zero eigenvalues of $A^TA$ (with multiplicities). 
\end{lemma}

\subsection{A general statement about discrete determinantal processes}\label{secgendet}

Let $P$ be an orthogonal projection matrix, where the rows and columns are both indexed by the set $H$ of size $m$. Let $r$ be the rank of $P$. Let $X$ be a random $r$-element subset of~$H$ such that for each deterministic $r$-element subset $A$ of $H$, we have
\[\mathbb{P}(X=A)=\det (P[A]).\]
Note that $P[A]$ is positive semidefinite, so $\det (P[A])\ge 0$. If we combine this with Corollary~\ref{projdetsum}, we see that such a random subset $X$ indeed exists. We call $X$ the determinantal process corresponding to $P$. See \cite{lyons2003determinantal,hough2006determinantal} for more information on determinantal processes.

Given an $r$-element subset $A$ of $H$, let $0\le \lambda_{A,1}\le\lambda_{A,2}\le\cdots\le\lambda_{A,r}$ be the eigenvalues of~$P[A]$. Note that if $\mathbb{P}(X=A)=\det (P[A])>0$, then these eigenvalues are all positive.

Although we will not need it, the spectrum of $P[A]$ has a nice probabilistic interpretation. Namely, $|X\cap A|$ has the same distribution as $\sum_{i=1}^r I_i$, where $I_1,\dots,I_r$ are independent Bernoulli random variables such that $\mathbb{E}I_i=\lambda_{A,i}$, see \cite{hough2006determinantal}.

\begin{lemma}\label{gendetspec}
For any $1\le k\le r$, we have
\[\mathbb{E}\sigma_k\left(\lambda_{X,1}^{-1},\lambda_{X,2}^{-1},\dots,\lambda_{X,r}^{-1}\right)\le {{m-r+k}\choose{k}}{{r}\choose{k}}\le \left(\frac{e^2mr}{k^2}\right)^k.\]
\end{lemma}
\begin{proof}
Combining the definition of $X$ with Corollary~\ref{sigmainv}, we obtain that
\begin{align*}
\mathbb{E}\sigma_k\left(\lambda_{X,1}^{-1},\lambda_{X,2}^{-1},\dots,\lambda_{X,r}^{-1}\right)&=\sum_{\substack{A\subset H,\,|A|=r\\\det P[A]>0} } \det (P[A])\sigma_k\left(\lambda_{A,1}^{-1},\lambda_{A,2}^{-1},\dots,\lambda_{A,r}^{-1}\right)\\&=\sum_{\substack{A\subset H,\, |A|=r\\\det P[A]>0 }}\quad\sum_{\substack{B\subset A\\|B|=r-k}} \det (P[B]).
\end{align*}
By a simple double counting argument and Corollary~\ref{projdetsum}, we have
\begin{align*}\sum_{\substack{A\subset H,\, |A|=r\\\det P[A]>0 }}\quad\sum_{\substack{B\subset A\\|B|=r-k}} \det (P[B])&\le \sum_{\substack{B\subset H\\|B|=r-k}} |\{A\subset H\,:\,B\subset A,\, |A|=r\}|\det (P[B])\\&={{m-r+k}\choose {k}}\sum_{\substack{B\subset H\\|B|=r-k}} \det (P[B])\\&={{m-r+k}\choose {k}}{{r}\choose {k}}. 
\end{align*}
Finally, relying on the estimate ${{\ell}\choose{k}}\le \left(\frac{e\ell}{k}\right)^k$, we get that
\[{{m-r+k}\choose {k}}{{r}\choose {k}}\le {{m}\choose {k}}{{r}\choose {k}}\le \left(\frac{e^2mr}{k^2}\right)^k. \qedhere\]

\end{proof}

\subsection{Determinantal hypertrees as determinantal processes}
\label{secdethyp}
Let $\partial_{n,d}$ be the matrix of the $d$th boundary map of the simplex on the vertex set $[n]$. As in the introduction, let $P_n$ be the orthogonal projection to the range of $\partial_{n,d}^T$, that is, the set of $d$-dimensional coboundaries of the simplex on the vertex set $[n]$.
Note that the rows and columns of $P_n$ are indexed by ${{[n]}\choose{d+1}}=\{A\subset [n]\,:\,|A|=d+1\}$.

Given a simplicial complex $C$, let $C(k)$ be the set of $k$-dimensional faces of $C$.

The next lemma was proved in \cite[Lemma 2.7 and Lemma 2.10]{meszaros2022local}.
\begin{lemma}\hfill\label{lemmahypertreeisdet}
\begin{enumerate}[(a)]
\item We have
\[P_n=\frac{1}n \partial_{n,d}^T \partial_{n,d}.\]
\item The rank of $P_n$ is ${{n-1}\choose {d}}$.

\item The determinantal process corresponding to $P_n$ is $T_n(d)$. In other words, for any ${{n-1}\choose {d}}$ element subset $A$ of ${{[n]}\choose{d+1}}$, we have
\[\mathbb{P}(T_n(d)=A)=\det (P_n[A]).\]
\end{enumerate}
\end{lemma}

As before, given a ${{n-1}\choose {d}}$ element subset $A$ of ${{[n]}\choose{d+1}}$, let $0\le \lambda_{A,1}\le\lambda_{A,2}\le\cdots\le\lambda_{A,{{n-1}\choose {d}}}$ be the eigenvalues of $P_n[A]$.

\begin{lemma}\label{nearzero}
For all $n\ge d+1$, we have
\[\sum_{k=0}^{{{n-1}\choose {d}}} \mathbb{P}\left(\sum_{i=1}^{k} \log\left(\frac{1}{n\lambda_{T_n(d),i}}\right)\ge n+2k\log\left(\frac{e{{n}\choose{d}}}{k}\right)\right)\le 2e^{-n}.\]
(Here for $k=0$, we interpret $0\log(\infty)$ as $0$.)
\end{lemma}
\begin{proof}

Let $k\ge 1$. By Lemma~\ref{lemmahypertreeisdet} and Lemma~\ref{gendetspec}, we see that
\[\mathbb{E} \prod_{i=1}^{k} \lambda_{T_n(d),i}^{-1}\le \mathbb{E} \sigma_{k}\left(\lambda_{T_n(d),1}^{-1},\lambda_{T_n(d),2}^{-1},\dots,\lambda_{T_n(d),{{n-1}\choose{d}}}^{-1}\right)\le \left(\frac{e^2 {{n}\choose {d+1}}{{n}\choose {d}}}{k^2}\right)^{k}. \]
 Combining this with Markov's inequality, we obtain that
 \begin{align*}
 \mathbb{P}&\left(\sum_{i=1}^{k} \log\left(\frac{1}{n\lambda_{T_n(d),i}}\right)\ge n+2k\log\left(\frac{e{{n}\choose{d}}}{k}\right)\right)\\&=
 \mathbb{P}\left(\prod_{i=1}^{k} \lambda_{T_n(d),i}^{-1}\ge e^n n^{k} \left(\frac{e{{n}\choose{d}}}{k}\right)^{2k}\right)\\&\le e^{-n}n^{-k} \left(\frac{k}{e{{n}\choose{d}}}\right)^{2k} \mathbb{E} \prod_{i=1}^{k} \lambda_{T_n(d),i}^{-1}\\&\le e^{-n} \left(\frac{ {{n}\choose {d+1}}}{n{{n}\choose {d}}}\right)^{k} \\&\le
 e^{-n}(d+1)^{-k}.
 \end{align*}
The above inequality is trivially also true for $k=0$. Summing these inequalities over $k$ the statement follows. 
\end{proof}

\subsection{The local weak limit of determinantal hypertrees}

Given a $d$-dimensional hypertree $S_n$ on the vertex set $[n]$, we define the bipartite graph $G_n$ as follows: The two color classes of $G_n$ are $S_n(d)$ and $S_n(d-1)={{[n]}\choose{d}}$. A $d$-dimensional face $A\in S_n(d)$ is connected to a $(d-1)$-dimensional face $B\in S_n(d-1)$ in the graph $G_n$ if and only if $B\subset A$.

Let $(F,o)$ be finite rooted tree, that is, $F$ is a finite tree, and $o$ is a distinguished vertex of $F$ called the root. Let $r$ be the depth of $(F,o)$, that is, the maximum of the distances of the vertices from the root.

Given a vertex $o_n$ of $G_n$, let $B_r(G_n,o_n)$ be the $r$-neighborhood of $o_n$, that is, the subgraph of $G_n$ induced by all the vertices at distance at most $r$ from the root $o_n$. By choosing $o_n$ as the root, $B_r(G_n,o_n)$ can be considered as a rooted graph.

Let us define
\[t(S_n,(F,o))=\frac{1}{{{n}\choose{d}}}\left|\left\{o_n\in {{[n]}\choose{d}}\,:\, B_r(G_n,o_n)\cong (F,o)\right\}\right|,\]
where $B_r(G_n,o_n)\cong (F,o)$ refers to an isomorphism of rooted graphs, that is, we are looking for a graph isomorphism which also preserves the root.

In the paper~\cite{meszaros2022local}, a random infinite rooted tree $(\mathbb{T}_d,o)$ called the semi-$d$-ary skeleton tree was defined, and it was proved that the (quenched) local weak limit of determinantal hypertrees is given by this random rooted tree $(\mathbb{T}_d,o)$. More precisely, we have the following theorem:
\begin{theorem}\label{thmlocal}
 For any finite rooted tree $(F,o)$ of depth $r$, we have
 \[\lim_{n\to\infty} t(T_n,(F,o))=\mathbb{P}(B_r(\mathbb{T}_d,o)\cong (F,o))\]
 in probability.
\end{theorem}

\begin{remark}\label{remark1out}
 Linial and Peled~\cite{linial2019enumeration} determined the local weak limit of $1$-out $d$-complexes. It turns out that these complexes have the same local weak limit as determinantal hypertrees~\cite{meszaros2022local}. This implies that certain results proved by Linial and Peled~\cite{linial2019enumeration} for $1$-out $d$-complexes are also true for determinantal hypertrees. 
\end{remark}

\subsection{Simple bounds on the size of the homology group}
\begin{lemma}\label{cdlower}
Let $0<c<\sqrt{\frac{d+1}e}$. Then
\[\lim_{n\to\infty}\mathbb{P}\left(|H_{d-1}(T_n,\mathbb{Z})|\le c^{{{n}\choose {d}}}\right)=0.\]
\end{lemma}
\begin{proof}
Let $\mathcal{T}(n)$ be the set of $d$-dimensional hypertrees on the vertex set $[n]$, and let
\[\mathcal{T}_0(n)=\left\{S\in \mathcal{T}(n)\,:\,|H_{d-1}(S,\mathbb{Z})|\le c^{{{n}\choose {d}}}\right\}.\]

Clearly,
\[|\mathcal{T}_0(n)|\le |\mathcal{T}(n)|\le {{{n}\choose{d+1}}\choose{{n-1}\choose{d}}}\le \left(\frac{e{{n}\choose{d+1}}}{{{n-1}\choose{d}}}\right)^{{n-1}\choose{d}}=\left(\frac{en}{d+1}\right)^{{n-1}\choose{d}}. \]

Thus,
\begin{align*}\mathbb{P}\left(|H_{d-1}(T_n,\mathbb{Z})|\le c^{{{n}\choose {d}}}\right)&=\sum_{S\in \mathcal{T}_0(n)} \frac{|H_{d-1}(S,\mathbb{Z})|^2}{n^{{{n-2}\choose{d}}}}\\&\le \frac{|\mathcal{T}_0(n)| c^{2{{n}\choose{d}}}}{n^{{{n-2}\choose{d}}}}\\&\le
\left(\frac{en}{d+1}\right)^{{n-1}\choose{d}} c^{2{{n}\choose{d}}}n^{-{{n-2}\choose{d}}}
\\&=\left(\frac{ec^2}{d+1}\right)^{{{n}\choose{d}}} \left(\frac{en}{d+1}\right)^{-{{n-1}\choose{d-1}}}n^{{{n}\choose{d}}-{{n-2}\choose{d}}}\\&=\left(\frac{ec^2}{d+1}\right)^{{{n}\choose{d}}} \exp\left(O\left(n^{d-1}\log (n)\right)\right).
\end{align*}
One can easily prove that the bound above tends to $0$ by relying on the fact that $\frac{ec^2}{d+1}<1$.
\end{proof}
Somewhat weaker versions of Lemma~\ref{cdlower} were proved in~\cite{kahle2022topology}.

Kalai~\cite{kalai1983enumeration} proved the following estimate.
\begin{lemma}\label{cdupper}
For any $d$-dimensional hypertree $S_n$ on $n$ vertices, we have
\[|H_{d-1}(S_n,\mathbb{Z})|\le \sqrt{d+1}^{{n-1}\choose{d}}.\]
\end{lemma}

\subsection{Nice sequences of hypertrees}

For each $n\ge d+1$, let $S_n$ be a $d$-dimensional (deterministic) hypertree on the vertex set $[n]$. We say that $S_{d+1},S_{d+2},\dots$ is a nice sequence of hypertrees if 
\begin{enumerate}[(C1)]
 \item For all large enough $n$, for any choice of $0\le k\le {{n-1}\choose{d}}$, we have
 \[\sum_{i=1}^{k} \log\left(\frac{1}{n\lambda_{S_n(d),i}}\right)\le n+2k\log\left(\frac{e{{n}\choose{d}}}{k}\right).\]
 \item\label{C2} For any finite rooted tree $(F,o)$ of depth $r$, we have
 \[\lim_{n\to\infty} t(S_n,(F,o))=\mathbb{P}(B_r(\mathbb{T}_d,o)\cong (F,o)).\]
\end{enumerate}

\begin{lemma}\label{nicecoupling}
There is a coupling of $T_{d+1},T_{d+2},\dots$ such that $T_{d+1},T_{d+2},\dots$ is a nice sequence of hypertrees with probability $1$.
\end{lemma}
\begin{proof}
 Let $((F_i,o_i))_{i=1}^\infty$ be an enumeration of all the finite rooted trees (up to isomorphism). Let $r_i$ be the depth of $(F_i,o_i)$. Then, we define
 \begin{align*}\Delta_i(T_n)&= \left|t(T_n,(F_i,o_i))-\mathbb{P}(B_{r_i}(\mathbb{T}_d,o)\cong (F_i,o_i))\right|,\text{ and }\\
 \Delta(T_n)&=\sum_{i=1}^\infty 2^{-i} \Delta_i(T_n).
 \end{align*}
 We claim that $\Delta(T_n)$ converges to $0$ in probability. To prove this, let $\ell$ be any positive integer. By Theorem~\ref{thmlocal}, \[\mathbb{P}\left(\Delta_i(T_n)>2^{-\ell}\right)<\frac{2^{-\ell}}{\ell}\text{ for all }i=1,\dots,\ell\]
 provided that $n$ is large enough. Thus, for all large enough $n$, we have
 \[\mathbb{P}\left(\Delta_i(T_n)\le 2^{-\ell}\text{ for all }i=1,\dots,\ell\right)>1-2^{-\ell}.\]
 On the event above, we have
 \[\Delta(T_n)=\sum_{i=1}^\ell 2^{-i}\Delta_i(T_n)+\sum_{i=\ell+1}^\infty2^{-i}\Delta_i(T_n)\le \sum_{i=1}^\ell 2^{-i}\cdot 2^{-\ell}+\sum_{i=\ell+1}^\infty 2^{-i}\le 2^{-(\ell-1)}.\]
 Thus, for all large enough $n$, $\mathbb{P}(\Delta(T_n)>2^{-(\ell-1)})< 2^{-\ell}$. Therefore, $\Delta(T_n)$ indeed converges to $0$ in probability. 

 Let $\mathcal{T}(n)$ be the set of $d$-dimensional hypertrees on the vertex set $[n]$. For each $n\ge d+1$, there is a map $g_n:(0,1]\to \mathcal{T}(n)$, with the following properties:
 \begin{itemize}
     \item For all $S\in\mathcal{T}(n)$, the preimage $g_n^{-1}(S)$ is an interval $(a_S,b_S]$ such that \[b_S-a_S=\mathbb{P}(T_n=S).\]
     \item The function $x\mapsto \Delta(g_n(x))$ is monotonically decreasing.
 \end{itemize}

Let $U$ be a uniform random element of $(0,1]$, then $g_{d+1}(U),g_{d+2}(U),\dots$ is a coupling of $T_{d+1},T_{d+2},\dots$. We show that this coupling has the desired properties. Fix $x\in (0,1]$. Let $\varepsilon>0$. Then $\Delta(g_n(x))>\varepsilon$ if and only if $x\le \mathbb{P}(\Delta(T_n)>\varepsilon)$. Since $\Delta(T_n)$ converges to $0$ in probability, we have $\lim_{n\to\infty}\mathbb{P}(\Delta(T_n)>\varepsilon)=0$. So $\Delta(g_n(x))\le \varepsilon$ for all sufficiently large $n$. Thus, $\lim_{n\to\infty} \Delta(g_n(x))=0$. Then for any positive integer $i$, we have
\[\lim_{n\to\infty} \Delta_i(g_n(x))\le 2^i \lim_{n\to\infty} \Delta(g_n(x))=0.\]
Therefore, (C2) holds for the coupling $g_{d+1}(U),g_{d+2}(U),\dots$. 
 
 Combining Lemma~\ref{nearzero} with the Borel--Cantelli lemma, we see that (C1) also holds.
\end{proof}

\section{Deterministic results}\label{detsection}

 Theorem~\ref{thm1} follows by combining Lemma~\ref{cdlower}, Lemma~\ref{cdupper}, Lemma~\ref{nicecoupling} and the following deterministic statement.

\begin{lemma}\label{fornice}
 Let $S_{d+1},S_{d+2},\dots$ be a nice sequence of $d$-dimensional hypertrees. Then
 \[\lim_{n\to\infty} \frac{\log |H_{d-1}(S_n,\mathbb{Z})|}{{{n}\choose{d}}}=c_d,\]
 for some constant $c_d$ (depending only on $d$).
\end{lemma}

The rest of Section~\ref{detsection} is devoted to the proof of Lemma~\ref{fornice}. Thus, throughout this section, we assume that $S_{d+1},S_{d+2},\dots$ is a nice sequence of $d$-dimensional hypertrees.

\subsection{The homology growth in terms of spectral measures}

Let $\partial_{S_n,d}$ be the matrix of the $d$-th boundary map of $S_n$, and let us consider the Laplacian matrix
\[L_n=\partial_{S_n,d}\partial_{S_n,d}^T.\]
Note the the rows and columns of $L_n$ are indexed by ${{[n]}\choose{d}}$. Let $\pi_n$ be the product of the nonzero eigenvalues of $L_n$, and let $\mu_n$ be the empirical measure on the eigenvalues of $L_n$. 

Observe that
\[P_n[S_n(d)]=\frac{1}{n}\partial_{S_n,d}^T\partial_{S_n,d}.\]

Combining this with Lemma~\ref{switch}, we see that $\mu_n$ can be expressed in terms of the eigenvalues of $P_n[S_n(d)]$ as
\[\mu_n=\frac{1}{{{n}\choose{d}}}\left({{n-1}\choose{d-1}}\delta(0)+\sum_{i=1}^{{{n-1}\choose{d}}}\delta\left(n\lambda_{S_n(d),i}\right)\right),\]
where $\delta(x)$ is the Dirac measure on $x$.

The significance of the measure $\mu_n$ is clear from the next lemma.
\begin{lemma}\label{Hvslog}
We have
\[\lim_{n\to\infty}\left(\frac{\log |H_{d-1}(S_n,\mathbb{Z})|}{{{n}\choose{d}}}-\frac{1}2\int \mathbbm{1}(t>0)\log(t)d\mu_n(t)\right)=0.\]
\end{lemma}
\begin{proof}
 Combining \cite[Theorem 1.3 (1)]{duval2009simplicial} with Kalai's formula given in \eqref{kalaiformula}, we obtain that
 \[\pi_n=n^{{{n-2}\choose{d-1}}}|H_{d-1}(S_n,\mathbb{Z})|^2.\]

 Thus,
 \[\frac{\log |H_{d-1}(S_n,\mathbb{Z})|}{{{n}\choose{d}}}=\frac{1}2\frac{\log(\pi_n)}{{{n}\choose{d}}}-O\left(\frac{\log (n)}n\right)=\frac{1}2\int \mathbbm{1}(t>0)\log(t)d\mu_n(t)-O\left(\frac{\log (n)}n\right).\]

 Therefore, the statement follows.
\end{proof}

\subsection{The weak converge of the empirical spectral measures}
The following lemma is a consequence of the work of Linial and Peled~\cite[Lemma 3.2.]{linial2016phase} and the assumption (C2). See also \cite{linial2019enumeration} and Remark~\ref{remark1out}.
\begin{lemma}\label{weakconvmeasure}
 There is a measure $\mu$ (depending only on $d$) such that $\mu_n$ converges to $\mu$ weakly. Thus, by definition, for any bounded continuous function $f$, we have
 \[\lim_{n\to\infty} \int f(t) d\mu_n(t)=\int f(t) d\mu(t).\]
\end{lemma}

For completeness, based on Linial and Peled~\cite{linial2016phase,linial2019enumeration}, we sketch the construction of the limiting measure $\mu$, but this construction will not be used in this paper.

Consider an infinite rooted tree $(T,o)$ such that all the degrees are finite. Note that $T$ is bipartite. Let $V$ be the set of vertices of $T$, which are in the same color class as $o$, that is, which are at even distance from $o$. Let $L_T$ be the operator which is densely-defined on the subset of finitely supported functions of the Hilbert space $\ell^2(V)$ by the equation that for all $u,v\in V$, we have
\[\langle L_T e_v,e_u\rangle=\begin{cases}
\deg(v)&\text{if $u=v$},\\
1&\text{if $u$ and $v$ are at distance $2$},\\
0&\text{otherwise},
\end{cases}\]
where $e_v$ is the characteristic vector of $v$. Assuming that $L_T$ is essentially self-adjoint, one can define $f(L_T)$ for any measurable function $f:\mathbb{R}\to\mathbb{C}$, and there is a unique probability measure $\mu_{(T,o)}$ such that
\[\langle f(L_T) e_o,e_o\rangle=\int f(t) d\mu_{(T,o)}(t)\]
for any measurable function $f:\mathbb{R}\to\mathbb{C}$.

Note that for the semi-$d$-ary skeleton tree $(\mathbb{T}_d,o)$, $L_{\mathbb{T}_d}$ is essentially self-adjoint almost surely, see~\cite{linial2016phase,linial2019enumeration} and Remark~\ref{remark1out}. Thus, one can define the measure
\[\mu=\mathbb{E}\mu_{(\mathbb{T}_d,o)}.\]

By \cite[Lemma 4.2]{linial2019enumeration} and Remark~\ref{remark1out}, we have the following lemma.
\begin{lemma}\label{noatomat0}
 We have $\mu(\{0\})=0.$
\end{lemma}

\subsection{Integrating the logarithm function}

By Lemma~\ref{Hvslog}, we see that Lemma~\ref{fornice} follows from the following lemma.

\begin{lemma}\label{lemmalog}
We have
\[\lim_{n\to\infty}\int \mathbbm{1}(t>0)\log(t)d\mu_n(t)=\int \mathbbm{1}(t>0)\log(t)d\mu(t).\]
In particular, the integral on the right exist.
\end{lemma}
Lemma~\ref{lemmalog} is a straightforward consequence of Lemma~\ref{lemmalog0} and Lemma~\ref{lemmaloginf} below.
\begin{lemma}\label{lemmalog0}
We have
\[\lim_{n\to\infty}\int \mathbbm{1}(1\ge t>0)\log(t)d\mu_n(t)=\int \mathbbm{1}(1\ge t>0)\log(t)d\mu(t).\]
Moreover, the integral on the right is finite.
\end{lemma}

\begin{proof}
 Given $0<\gamma<1$, let us define the function $\ell_\gamma$ by
 \[\ell_\gamma(t)=\begin{cases}
 0&\text{for }t\le 0,\\
 \frac{t}{\gamma}\log(\gamma)&\text{for }0<t\le \gamma,\\
 \log(t)&\text{for }\gamma<t\le 1,\\
 0&\text{for }1<t.
 \end{cases}\]
 It is straightforward to see that $\ell_\gamma$ is a bounded continuous function. Thus, by Lemma~\ref{weakconvmeasure}, we have
 \begin{equation}\label{ellgammacont}
 \lim_{n\to\infty}\int \ell_\gamma(t)d\mu_n(t)=\int \ell_\gamma(t)d\mu(t).
 \end{equation}

 Let 
 \[k_n=\max\{i\,:\,n\lambda_{S_n(d),i}\le \gamma\}\quad\text{ and }\quad\varepsilon=\limsup_{n\to\infty}\frac{k_n}{{{n}\choose{d}}}.\]

 By Lemma~\ref{weakconvmeasure}, we have
 \begin{equation}\label{varepsilonbound}\varepsilon\le \limsup_{n\to\infty} \mu_n([0,\gamma])\le \mu([0,\gamma]).\end{equation}

 Using condition~(C1) of the definition of nice sequences, for all large enough $n$, we have 
 \begin{align*}
 0\le \int \ell_\gamma(t) d\mu_n(t)-\int \mathbbm{1}(1\ge t>0)\log(t) d\mu_n(t)&\le \frac{1}{{{n}\choose{d}}}\sum_{i=1}^{k_n} \log\left(\frac{1}{n\lambda_{S_n(d),i}}\right)\\&\le \frac{n}{{{n}\choose{d}}}+2\frac{k_n}{{{n}\choose{d}}}\log\left(\frac{e{{n}\choose{d}}}{k_n}\right).
 \end{align*}
 If we combine this with \eqref{varepsilonbound} and the fact that $x\mapsto 2x\log\left(\frac{e}x\right)$ is an increasing continuous function on $[0,1]$, we obtain that
 \[\limsup_{n\to\infty}\left|\int \ell_\gamma(t) d\mu_n(t)-\int \mathbbm{1}(1\ge t>0)\log(t) d\mu_n(t)\right|\le 2\mu([0,\gamma])\log\left(\frac{e}{\mu([0,\gamma])}\right).\]
 Combining this with \eqref{ellgammacont}, we get that
 \begin{equation}\label{almostfinal}\limsup_{n\to\infty}\left|\int \ell_\gamma(t) d\mu(t)-\int \mathbbm{1}(1\ge t>0)\log(t) d\mu_n(t)\right|\le 2\mu([0,\gamma])\log\left(\frac{e}{\mu([0,\gamma])}\right).\end{equation}

 Using Lemma~\ref{noatomat0}, we see that $\mu([0,\gamma])$ tends to $0$ as $\gamma$ tends $0$. Therefore, the right hand side of \eqref{almostfinal} converges to $0$ as $\gamma$ tends to $0$. Thus, it follows easily from~\eqref{almostfinal} that
 \begin{equation}\label{veryfinal}\lim_{n\to\infty}\int \mathbbm{1}(1\ge t>0)\log(t)d\mu_n(t)=\lim_{\gamma\to 0} \int \ell_\gamma(t)d\mu(t)=\int \mathbbm{1}(1\ge t>0)\log(t)d\mu(t),\end{equation}
 where the last equality follows from the monotone convergence theorem. 
 
 Applying \eqref{almostfinal} with the choice of $\gamma=\frac{1}2$, we see that for every large enough $n$, we have
 \begin{align*}\liminf_{n\to\infty}\int \mathbbm{1}(1\ge t>0)\log(t)d\mu_n(t)&\ge \int \ell_{1/2}(t)d\mu(t)-2\mu([0,1/2])\log\left(\frac{e}{\mu([0,1/2])}\right)\\&\ge \log(1/2)-2\mu([0,1/2])\log\left(\frac{e}{\mu([0,1/2])}\right).\end{align*}

Combining this with \eqref{veryfinal}, we see that
\[\int \mathbbm{1}(1\ge t>0)\log(t)d\mu(t)>-\infty.\]

Thus, the statement follows.
\end{proof}

\begin{lemma}\label{lemmaloginf}
We have
\[\lim_{n\to\infty}\int \mathbbm{1}(t>1)\log(t)d\mu_n(t)=\int \mathbbm{1}(t>1)\log(t)d\mu(t).\]
Moreover, the integral on the right is finite.
\end{lemma}
\begin{proof}
Given $\omega>e$, let us define the function $h_\omega$ by
 \[h_\omega(t)=\begin{cases}
 0&\text{for }t\le 1,\\
 \log(t)&\text{for }1<t\le \omega,\\
 \log(\omega)&\text{for }\omega<t.
 \end{cases}\]
 It is straightforward to see that $h_\omega$ is a bounded continuous function. Thus, by Lemma~\ref{weakconvmeasure}, we have
 \begin{equation}\label{homegacont}
 \lim_{n\to\infty}\int h_\omega(t)d\mu_n(t)=\int h_\omega(t)d\mu(t).
 \end{equation}

Let 
 \[B_n=\{i\,:\,n\lambda_{S_n(d),i}\ge \omega\}\quad\text{ and }\quad b_n=|B_n|.\]

Note that every diagonal entry of $P_n$ is $\frac{d+1}n$. Thus,
 \begin{equation}\label{sumestimate}\sum_{i\in B_n} n\lambda_{S_n(d),i}\le \sum_{i=1}^{{{n-1}\choose{d}}} n\lambda_{S_n(d),i}=n\Tr(P_n[S_n(d)])=(d+1){{n-1}\choose{d}}\le (d+1){{n}\choose{d}}.\end{equation}
 
Consequently,
 \begin{equation}\label{bnestimate}b_n\le \frac{1}{\omega} \sum_{i\in B_n} n\lambda_{S_n(d),i}\le \frac{d+1}{\omega}{{n}\choose{d}}.\end{equation}

Combining Jensen's inequality with \eqref{sumestimate},
 \begin{equation}\label{Jensen}\frac{1}{{{n}\choose{d}}}\sum_{i\in B_n}\log( n\lambda_{S_n(d),i})\le \frac{b_n}{{{{n}\choose{d}}}}\log\left(\frac{1}{b_n} \sum_{i\in B_n} n\lambda_{S_n(d),i}\right)\le \frac{b_n}{{{n}\choose{d}}} \log\left(\frac{(d+1){{n}\choose{d}}}{b_n} \right).\end{equation}

The function $x\mapsto x\log\left(\frac{d+1}{x}\right)$ is monotone increasing on $\left[0,\frac{d+1}{e}\right]$. Thus, combining \eqref{Jensen} and \eqref{bnestimate} with the assumption that $\omega>e$, we obtain
\begin{equation*}\frac{1}{{{n}\choose{d}}}\sum_{i\in B_n}\log( n\lambda_{S_n(d),i})\le \frac{d+1}{\omega} \log(\omega).\end{equation*}

Therefore,
 \[\left|\int \mathbbm{1}(t>1)\log(t) d\mu_n(t)-\int h_\omega(t) d\mu_n(t)\right|\le \frac{1}{{{n}\choose{d}}}\sum_{i\in B_n}\log( n\lambda_{S_n(d),i})\le \frac{d+1}{\omega} \log(\omega).\]
 
Combining this with \eqref{homegacont}, we get\begin{equation*}\limsup_{n\to\infty}\left|\int \mathbbm{1}(t>1)\log(t) d\mu_n(t)-\int h_\omega(t) d\mu(t)\right|\le \frac{d+1}{\omega} \log(\omega).\end{equation*}

Since $\lim_{\omega\to\infty} \frac{d+1}{\omega} \log(\omega)=0$, it follows easily that
 \[\lim_{n\to\infty}\int \mathbbm{1}( t>1)\log(t)d\mu_n(t)=\lim_{\omega\to \infty} \int h_{\omega}(t)d\mu(t)=\int \mathbbm{1}(t>1)\log(t)d\mu(t)<\infty,\]
 where the last equality follows from the monotone convergence theorem.
\end{proof}

\bibliography{references}
\bibliographystyle{plain}

Andr\'as M\'esz\'aros

{\tt meszaros@renyi.hu}

HUN-REN Alfr\'ed R\'enyi Institute of Mathematics,

Budapest, Hungary

\end{document}